\documentclass[a4paper, 12pt]{amsart}
\usepackage{amsfonts, amsthm, amssymb, amsmath, stmaryrd}
\usepackage{mathrsfs,array}
\usepackage{xy}
\usepackage{verbatim}
\usepackage{amscd} 
\usepackage{tikz}
\usepackage{tikz-cd}
\usetikzlibrary{matrix,arrows,decorations.pathmorphing}
\usepackage{textcomp}
\usepackage{mathtools}
\input xy
\xyoption{all}
\usepackage[unicode]{hyperref}

\usepackage{etoolbox}
\makeatletter
\patchcmd{\@bibitem}{\ignorespaces}{\label{bib-#1}\ignorespaces}{}{}
\makeatother

\usepackage[margin = 1in]{geometry}

\newcommand{\rk}{\mathrm{rk}}

\newtheorem{theorem}{Theorem}

\newtheorem{lemma}{Lemma}

\newcommand{\Cl}{\mathrm{Cl}}

\author{Manjul Bhargava}
\address{Department of Mathematics, Princeton University, Princeton, NJ, USA}
\email{bhargava@math.princeton.edu}
\author{Arul Shankar}
\address{Department of Mathematics, University of Toronto, Toronto, ON, Canada}
\email{ashankar@math.toronto.edu}
\author{Artane Siad}
\address{Yau Mathematical Sciences Center, Tsinghua University, Beijing, China}
\email{ajsiad@mail.tsinghua.edu.cn}
\author{Ashvin Swaminathan}
\address{Department of Mathematics, Harvard University, Cambridge, MA, USA}
\email{swaminathan@math.harvard.edu}

\date{\today}

\title[Cubic fields with class group of exact $2$-rank $1$]{The existence of infinitely many cubic fields with class group of exact $2$-rank $1$}

\begin{document}

\maketitle

\begin{abstract}
We show that infinitely many cubic fields have class group of $2$-rank $1$.
\end{abstract}

\section{Introduction}

For any prime $p \neq 3$ and finite abelian $p$-group $A$, the Cohen--Lenstra--Martinet--Malle heuristic predicts that $\Cl_K[p^\infty] = A$ for infinitely many cubic fields $K$.
They propose in particular that $\rk_p \, \Cl_K$, the number of non-trivial cyclic factors in the $p$-part of the class group, takes each of the values $r = 0,1,2,3,\ldots$ infinitely often. 
The only proven case of this latter prediction is $(p,r) = (2,0)$, which was proven by the first-named author in his work \cite[Theorem 5]{MR2183288} on densities of discriminants of quartic rings.

This note resolves the case $(p,r) = (2,1)$.

\begin{theorem} \label{thm: main theorem}
    There are infinitely many cubic fields with class groups of exact $2$-rank $1$.
\end{theorem}

The ingredients of two different proofs of Theorem \ref{thm: main theorem} are essentially contained in the previous works \cite{MR4891959} and \cite{2506.05539}, and Theorem \ref{thm: main theorem} follows from them through a combination of favourable numerology, a multiplicity-of-monogeniser argument, and, for the second, considerations of the closed convex hull of $\{(2^n,2^{2n}) \colon n \in \mathbb{N}\} \subseteq \mathbb{R}^2$ spelled out in \S \ref{sec: proof}. Of the nine distinct stable families of $(a,d)$-monogenised\footnote{By $(a,d)$-monogenised cubic fields, for $a,d \in \mathbb{Z}$, we mean those of the form $K = \mathbb{Q}[x]/(f(x))$ where $f(x) = ax^3 + bx^2 + cx + d$ with $b,c \in \mathbb{Z}$ and for which the order $R_f = \langle 1, ax, ax^2+bx \rangle_\mathbb{Z}$ is maximal in $K$. The unit-monogenised fields mentioned later correspond to the case $(a,d) = (1,1)$.} cubic fields identified in \cite{MR4891959}, a single one has the configuration of a large-enough uniform lower bound on $\lvert \Cl[2] \rvert$ and a small-enough upper bound on average $\lvert \Cl[2] \rvert$ required for our argument to go through; while in \cite{2506.05539}, only real monogenised fields give a first/second $\lvert \Cl[2] \rvert$ moment vector lying close enough to the boundary of $\overline{\mathrm{co}}\{(2^n,2^{2n}) \colon n \in \mathbb{N}^+\}$; complex ones do not. 

Constructing number fields with class groups of large $p$-rank in the absence of genus theory\footnote{Genus theory makes these kinds of problems easy. Consider, for example, the case of imaginary quadratic fields. There, $\rk_2 \,\mathbb{Q}(\sqrt{d})$ with $d < 0$ square-free is completely determined by $d$: It is $\omega(\Delta_d)-1$ with $\Delta_d = d$ or $4d$ if $d \equiv 1$ or $2,3 \pmod{4}$ respectively, and $\omega(\cdot)$ the number of distinct prime factors.}, and in particular constructing cubic fields with large $\rk_2 \, \Cl$, is an old and active area which has occupied mathematicians since at least the 1920s. In 1922, Nagell proved that infinitely many quadratic fields had a copy of $\mathbb{Z}/n$ in their class groups, kicking off a steady line of results reaching to the present \cite{MR3069394}, \cite{MR85301}, \cite{MR79613},  \cite{MR236144}, \cite{MR360518}, \cite{MR797674}, \cite{MR864264}, \cite{MR866122}, \cite{MR1875340}, \cite{MR2590188}, \cite{MR3749365}, \cite{MR4033251},
each exploiting special structures in thin families. For $2$-ranks of class groups of cubic fields, most relevant to Theorem \ref{thm: main theorem}, Nakano \cite{MR937193} showed that infinitely many had $\rk_2 \Cl \ge 6$, Kulkarni \cite{MR3749365} that this could be pushed to $\rk_2 \Cl \ge 8$, and Gillibert--Levin \cite{MR4033251} to $\rk_2 \Cl \ge 11$. The $2$-rank of the class group of cubic fields is also known to be connected to the $2$-Selmer group of elliptic curves by \cite{MR480416} and \cite{MR1432782}. 
In contrast, constructing infinitely many cubic fields with class groups of exact $2$-rank $1,2,3,\ldots$ is much more delicate because the extra structure responsible for large $\rk_2 \Cl$ needs to be shown not to beget more rank than needed. Before Theorem \ref{thm: main theorem}, nothing was known about the infinitude of cubic fields with class groups of fixed exact $2$-rank beyond rank $0$.

The two approaches to Theorem \ref{thm: main theorem} alluded to above may be descriptively called the \textit{moment approach} for the one based on \cite{2506.05539} and the \textit{anomaly approach} for the one based on \cite{MR4891959}. Let us explicate both.

Cohen--Lenstra distributions tend to be determined by their moments. The idea behind the moment approach is that knowledge of finitely many moments might be sufficient to establish Theorem \ref{thm: main theorem}. As we explain in \S \ref{sec: proof BSS}, Theorem \ref{thm: main theorem} follows once enough information about the average of the vector $(\lvert \Cl[2]\rvert,\lvert \Cl[2]\rvert^2)$ is found to place it outside of the closed convex hull $\overline{\mathrm{co}}(\{(2^n,2^{2n}): \mathbb{N}\} \setminus \{(2,4)\})$. \cite[Theorem 4]{BHSpreprint} and \cite[Theorem 1.2]{2506.05539} are enough to reach this conclusion in the family of \textit{monogenised} cubic fields, those with ring of integers $\mathbb{Z}[\alpha]$. 

In the anomaly approach through \cite{MR4891959}, our idea is to leverage anomalous class group behaviour in a thin family of cubic fields consisting of fields that we call \textit{unit-monogenised}, which are characterised by having ring of integers $\mathbb{Z}[\alpha]$ with $\alpha$ a unit. 
In \cite{MR4891959}, these are divided into six distinct congruence subfamilies for which we prove upper bounds\footnote{Upper bounds which are equalities conditional on the tail estimate \cite[Remark 5.5]{MR4891959}.} on average $\rvert \Cl[2] \vert$ which remain stable under further congruence conditions. 
Some of the bounds in these stable families are anomalously large when viewed through the lens of the Cohen--Lenstra heuristics.
One of these stable families, the one we call $+B_{1,1}^2$, has $\lvert \Cl[2] \rvert \ge 2$ for all but negligibly many members.
We show that the average $\lvert \Cl[2] \rvert$ in $+B_{1,1}^2$ is bounded above by $3$, giving a positive proportion of elements of $+B_{1,1}^2$ with class groups of exact $2$-rank $1$ and hence, by the argument of \S \ref{sec: proof SSS}, Theorem \ref{thm: main theorem}. 
Because $+B_{1,1}^2$ is thin the geometry-of-numbers method introduced in \cite{MR2183288} does not apply and we obtain the upper bound of $3$ in an unorthodox way: We deploy its framework, but with Davenport's lemma, its main engine, replaced with effective dynamics \cite{MR3025156}.

Both approaches yield asymptotic lower bounds for cubic fields of discriminant at most $X$ and class group of exact $2$-rank $1$: $X^{1/2}$ for \cite{MR4891959} and $X^{5/6}$ for \cite{2506.05539}; both are comprised entirely of real cubic fields.

\begin{theorem} \label{thm: quantitative main theorem}
    We have $\#\{K \, \colon \, \lvert \mathrm{Disc} \, K \rvert < X \, \textrm{ and } \, \rk_2 \, \Cl_K = 1 \} \gg X^{5/6}$.
\end{theorem}

We point out that the $X^{5/6}$ is less than the $X$ predicted by Cohen--Lenstra--Martinet--Malle and that the $X^{1/2}$ from \cite{MR4891959} could be improved by considering, and summing, the family $+B_{1,1}^2$ across all $(a,d)$-monogenised fields.

The idea of leveraging anomalous class group behaviour represents a conceptual departure from mainstream strategies aimed at establishing the existence form of Cohen--Lenstra. It has a clear advantage: it allows us to achieve with an upper bound on a first moment what, in the approach via \cite{2506.05539}, required both a lower bound on a first moment and an upper bound on a second moment.

We remark that Theorem \ref{thm: main theorem} also echoes, on the class group side, recent developments on the infinitude of elliptic curves of small positive ranks over $\mathbb{Q}$ by Zywina \cite{2502.01957}.
One might speculate that attacking the problem of the infinitude of cubic fields with class groups of fixed exact $2$-rank $>1$ might be facilitated by studying anomalous families. In fact, as explained in \cite{ShankarSiadSwaminathanNumericsandConjecturesUnitMonogenized}, we expect an average $\lvert \Cl[2] \rvert^2$ of $12$ in $+B_{1,1}^2$: Considering $\overline{\mathrm{co}}\{(2^n,2^{2n}) \colon n \in \mathbb{N}^+\}$, an upper bound of this strength and lower bound of $3$ on average $\lvert \Cl[2]\rvert$ in $+B_{1,1}^2$ would then give the counterpart of Theorem \ref{thm: main theorem} for exact $2$-rank $2$.

\section{Proofs} \label{sec: proof}

\subsection{Proof via \cite{MR4891959}} 
\label{sec: proof SSS}

Let $+B_{1,1}^2$ denote the family of cubic polynomials of the form $f(x) = x^3+ax^2+bx+1 \in \mathbb{Z}[x]$ that are maximal and for which $(a,b) \pmod{4} \in \{(4,4),(1,2),(2,1)\}$ and $a,b >0$.  
We will denote the discriminant of $f(x)$ by $\Delta(a,b)$ and speak interchangeably of forms in $+B_{1,1}^2$ and of their associated cubic fields $\mathbb{Q}[x]/(f(x))$. In what follows, when talking of averages and of proportions, we may order $+B_{1,1}^2$ either by symmetric height, $\max \{ \lvert a \rvert, \lvert b \rvert \}$, or by weighted height, $\max \{ \lvert a \rvert, \lvert b \rvert^{1/2} \}$.

We will need two lemmas for our proof of Theorems \ref{thm: main theorem} and \ref{thm: quantitative main theorem}.

\begin{lemma} \label{lem: finite repetition}
    Up to isomorphism, a cubic field can occur at most $60$ times as the field $\mathbb{Q}[x]/(f(x))$ associated to a maximal polynomial of the form $f(x) = x^3+ax^2+bx+1$ with $a,b \in \mathbb{Z}$. 
\end{lemma}
\begin{proof}
    Under translations $x \mapsto x + n$, a cubic polynomial of the form $f(x) = x^3+ax^2+bx+1$ has a constant coefficient of $1$ at most three times. By Bennett's effectivisation \cite{Ben2001}\cite[Theorem A.1]{2409.02627} of the Birch--Merriman Theorem \cite{MR0306119}, a cubic field has at most $20$ monogenisers up to translations $x \mapsto x + n$, giving the statement. 
\end{proof}

\begin{lemma}\label{lem: infinitely many have rk_2 = 1}
    There are infinitely many members of the family $+B_{1,1}^2$ that have $\rk_2 \, \Cl = 1$.
\end{lemma}
\begin{proof}
    By any effective form Hilbert irreducibility, \cite[Theorem 5]{MR4891959} and the argument of \cite[Theorem 2.8]{MR4891959}, at least $50\%$ of fields in $+B_{1,1}^2$ have $\lvert \Cl[2] \rvert = 2$. 
\end{proof}

\begin{proof}[Proof of Theorem \ref{thm: main theorem}]
Because infinitely many members of $+B_{1,1}^2$ have class groups of exact $2$-rank $1$ by Lemma \ref{lem: infinitely many have rk_2 = 1} and any cubic field can be repeated as a member of $+B_{1,1}^2$ at most finitely many times by Lemma \ref{lem: finite repetition}, there are infinitely many totally real unit-monogenised cubic fields of exact $2$-rank $1$ and Theorem \ref{thm: main theorem} follows. 
\end{proof}

\begin{proof}[Proof of Theorem \ref{thm: quantitative main theorem} with $X^{1/2}$]
    The discriminant of $x^3+ax^2+bx+1$ is $a^2 b^2 - 4 a^3  - 4 b^3  + 18 a b -27$. Thus, $H_{\mathrm{bal}}(a,b) < X^{1/4}$ implies that $\lvert \Delta(a,b) \rvert < X$. Since maximal forms of balanced height at most $X^{1/4}$ grow like a constant multiple of $X^{1/2}$ by \cite[Theorem 5.1]{MR4891959}, a positive proportion of which lie in $+B_{1,1}^2$, of which at least $50\%$ have $\rk_2 \, \Cl = 1$, we obtain Theorem \ref{thm: quantitative main theorem}.
\end{proof}

\subsection{Proof via \cite{2506.05539}} \label{sec: proof BSS} Let $\mathcal{F}_{1,+}$ denote the family of cubic polynomials of the form $f(x) = x^3+ax^2+bx+c \in \mathbb{Z}[x]$ which are maximal and totally real, up to the substitution $x \mapsto x+n$. We order $\mathcal{F}_{1,+}$ by $\max\{\lvert a^2 - 3b \rvert^3, \lvert -2a^3+9ab-27c \rvert^2/4 \}$. It was shown in \cite[Theorem 4]{BHSpreprint} that the average $\lvert \Cl[2] \rvert$ is equal to $3/2$
and in \cite[Theorem 1.2]{2506.05539} that the average $\lvert\Cl[2]\rvert^2$ over $\mathcal{F}_{1,+}$ is bounded above by $3$.

\begin{proof}[Proof of Theorems \ref{thm: main theorem} and \ref{thm: quantitative main theorem}]
Suppose that $0\%$ of the members of $\mathcal{F}_{1,+}$ have $\lvert \Cl[2]\rvert = 2$. We would then have a solution\footnote{For this argument it doesn't matter if the $p_i$ are constants or oscillate as $H \rightarrow \infty$.} $(p_1,p_4,p_8,\ldots)$ to the infinite system
 $$\begin{pmatrix}
    1 &4 &8 &16 &\cdots \\ 
    1 &4^2 &8^2 &16^2 &\cdots \\
\end{pmatrix} \begin{pmatrix}
    p_1\\
    p_4\\
    p_8\\
    \vdots
\end{pmatrix} = \begin{pmatrix}
    3/2\\
    \beta\\
\end{pmatrix}
$$
with $p_1 + p_4 + p_8 + \ldots = 1$ and $\beta \le 3$. Geometrically speaking: The closed convex hull $\overline{\mathrm{co}}(\{(2^n,2^{2n}): \mathbb{N}\} \setminus \{(2,4)\})$ contains the point $(3/2,\beta)$. This is impossible since the slope between $(1,1)$ and $(3/2,\beta)$ is strictly smaller than that between $(1,1)$ and $(4,4^2)$. We conclude that a positive proportion of members of $\mathcal{F}_{1,+}$ have $\lvert \Cl[2] \rvert = 2$, giving Theorem \ref{thm: main theorem} by Birch--Merriman \cite{MR0306119} and Theorem \ref{thm: quantitative main theorem} by \cite{2506.05539}. 
\end{proof}

For the family $\mathcal{F}_{1,-}$ of complex monogenised fields we cannot rule out that $0\%$ of the members of $\mathcal{F}_{1,-}$ have $\lvert \Cl[2] \rvert = 2$. Indeed, \cite{2506.05539} and \cite{BHSpreprint} give a moment vector of the form $ \begin{psmallmatrix} 2\\ \beta' \end{psmallmatrix}$ with $\beta' \le 6$, which could lie on the slope between $(1,1)$ and $(4,4^2)$.

This moment vector does, however, rule out simultaneous $0\%$ for both $\lvert \Cl[2] \lvert = 2$ and $\lvert \Cl[2] \lvert = 4$ and we find a dichotomy which also holds in $-B_{1,1}^2$ for simple reasons.

\begin{theorem}
    Infinitely many complex cubic fields have class groups of $2$-rank $1$ or $2$.\footnote{The reasoning of \S \ref{sec: proof BSS} and conjectures of \cite{SiadVenkatesh} suggest the existence of bounds on the second moment of $\lvert \Cl[2] \rvert$ over non-evenly ramified $n$-monogenised cubic fields which would establish Theorem \ref{thm: main theorem} for both signatures. Indeed, \cite{SiadVenkatesh} predicts a moment vector of $(3/2,3)$ for imaginary cubic fields of this type.}
\end{theorem}

\section*{Acknowledgements}

We will not repeat the names listed in the acknowledgements of \cite{MR4891959} and \cite{2506.05539}, who naturally deserve to be thanked again, but would like to give special recognition to Iman Setayesh who had several key discussions with the first and second named authors about the heart of \cite{MR1230290}, but chose not to be named as an author of \cite{MR4891959}, and to the anonymous referee of \cite{MR4891959}, whose comments led us to break up unit-monogenised cubic fields into stable families. We also thank Jordan Ellenberg for insightful questions.

\bibliographystyle{amsalpha}
\bibliography{references}

\end{document}